\documentclass[11pt, a4paper, oneside, reqno]{amsart}
\usepackage{makros}
\usepackage{graphicx}
\usepackage{amssymb,amsfonts,amsmath,amsthm}
\usepackage{stmaryrd,hyperref,yhmath}
\usepackage{color}
\allowdisplaybreaks
\numberwithin{equation}{section}
\newtheorem{result}[theorem]{Main result}
\newtheorem{remark}     [theorem]{Remark}

\begin{document}

\begin{abstract}
  We prove the existence of solitary waves in a lattice where all particles interact with each other by pair-wise repulsive forces that decay with distance.  The variational existence proof is based on constrained optimization and provides a one-parameter family of unimodal solutions. We also describe the asymptotic behavior of large, fast, high-energy waves.
\end{abstract}
\maketitle

\section{Introduction}

We consider infinite one-dimensional lattices of particles where each particle interacts with every  through pairwise forces. Our main focus is on the existence of solitary traveling wave solutions in generalized Fermi-Pasta-Ulam-Tsingou (FPUT) lattices where the dynamics are governed by the differential equations
\begin{align}
\label{eq:dynamic}
\ddot{x}_j = \suminf \big( \Phi'_m(x_{j+m}-x_j) - \Phi'_m(x_j-x_{j-m}) \big), \quad  j \in \Zset.
\end{align}
Here, $x_j \in \Rset$ denotes the position of particle $j$ with $x_j \leq x_{j+1}$ and $\Phi_m(r)$ describes the pairwise interaction potentials with the $m^{th}$ neighbor to the left and the right. A typical case are uniform lattices with  $\Phi_m(r)= r^{-\alpha}$ for suitable $\alpha > 0$.  We study solitary traveling waves which satisfy
\begin{align}
\label{eq:trav_ansatz}
x_j(t) = \nu j - X(j-ct),
\end{align}
where $\nu$ is a fixed parameter describing an equilibrium distribution of the lattice, $X$ is the wave profile and $c$ the wave speed.

There has been substantial recent interest in the study of traveling waves of long-range FPUT lattices \eqref{eq:dynamic}. Rigorous  small amplitude, long-wave solutions are given in  \cite{IP25} and \cite{AW25} through approximations with Benjamin-Ono and Korteweg–de Vries equations respectively. For initial data approximations by long-wave solutions of  Benjamin-Ono equations were studied in \cite{Wri24}. In this paper, we will prove the existence of traveling waves without restriction to small amplitudes for a wide class of long-range interactions.

\subsection{Reformulation}

We are setting $W = \dot{X}$, $\xi = j-ct$ and inserting the traveling wave ansatz \eqref{eq:trav_ansatz} into the dynamic equation \eqref{eq:dynamic} and we get
\begin{align}
\label{eq:wave_aux}
-c^2 \,\dot{W}(\xi) = \suminf \big( \Phi'_m(\nu m + X(\xi) - X(\xi + m)) - \Phi'_m(\nu m - X(\xi) + X(\xi - m) )\big).
\end{align}
Defining the operator $A_m$ by convolution with the characteristic function $\chi_{[-m/2,m/2]}$
\begin{align}
A_m f(\xi) := \chi_{[-m/2,m/2]} \ast f (\xi) =\int\limits_{-\frac{m}{2}}^{\frac{m}{2}} f(\xi+s) \, {\rm d} s,
\end{align}
we can express the terms in the sum in \eqref{eq:wave_aux} as
\begin{align}
    & \Phi'_m\big((\nu m + X(\xi) - X(\xi + m)\big) - \Phi'_m\big(\nu m - X(\xi) + X(\xi - m)\big)\nonumber \\
    &= \Phi'_m\big(\nu m - A_m W (\xi + \frac{m}{2})\big) - \Phi'_m\big(\nu m - A_m W (\xi - \frac{m}{2}) \big).
\end{align}
Integrating \eqref{eq:wave_aux} with respect to $\xi$ then yields
\begin{align}
\label{eq:wave_eq}
c^2 \, W(x) &= - \suminf \Big( \int_{-\infty}^x \Phi_m'\big(\nu m - A_m W(\xi + \frac{m}{2} ) \big) - \Phi_m'\big(\nu m - A_m W(\xi - \frac{m}{2} ) \big) \, {\rm d} \xi \Big) + \eta \nonumber \\
&= - \suminf \Big( \int_{x-m}^x \Phi_m'\big(\nu m - A_m W(\xi + \frac{m}{2} ) \big) \, {\rm d} \xi \Big) + \eta \nonumber \\
&= - \suminf \Big( \int_{-\frac{m}{2}}^{\frac{m}{2}} \Phi_m'\big(\nu m - A_m W(x + \xi ) \big) \, {\rm d} \xi \Big) + \eta\nonumber \\
&= \suminf \big( A_m (- \Phi_m'(\nu m - A_m W(x) ) ) \big)  + \eta,
\end{align}
due to the structure of the integrand and where $\eta$ is a suitable constant. We will prove the existence of solutions $(W,c)$ to \eqref{eq:wave_eq} to give the existence of traveling waves of speed $c$ via variational methods.

\subsection{Main Results}

Our approach requires certain properties on the potentials and the norm of  $W$:

{\bf Assumption on $W$:} We will construct our solution by a constrained optimization problem  with a constraint $\|W \|_2^2 = 2K$,  where
\begin{align}
\label{eq:Assumption_K}
0 < K < \frac{\nu^2}{2}
\end{align}
is required.  We note that this implies, with Young's convolution inequality, see Lemma \ref{le:AmW} below,
\begin{align}
    |X(\xi)-X(\xi-m)| = |A_m W|(\xi-\frac{m}{2})  \leq \sqrt{m} \sqrt{2K} < \sqrt{m} \nu.
\end{align}
Hence, we have $|X(\xi)-X(\xi-m)| < \nu m $ and the expressions in \eqref{eq:wave_aux} are well-defined for typical potentials like $r^{-\alpha}$. Indeed, this implies $x_j < x_{j+1}$ for the traveling wave solutions.

{\bf Assumptions on $\Phi_m$:} For every $m \in \Nset$, the potentials $\Phi_m: [0,\infty) \to [0,\infty)$ are  in $C^4([0,\infty))$ and satisfy
\begin{align}
\Phi_m(s) \geq 0, \quad \Phi_m'(s) \leq 0, \quad \Phi_m''(s) \geq 0, \quad \Phi_m'''(s) \leq 0, \quad \Phi_m^{(4)}(s) \geq 0 \label{eqn:assPhi}
\end{align}
for all $s \geq 0$.  We will require that these inequalities are strict for $m=1$ and $s>0$. Furthermore, we assume the global bounds
\begin{subequations}
\begin{align}
\label{eq:Phi_convergence}
&\suminf \Phi_m'(\nu m - \sqrt{2Km}) m < \infty, \\
&\suminf \Phi_m''(\nu m - \sqrt{2Km}) m^2 < \infty, \label{eq:Phi_convergence2} \\
&\suminf \Phi_m''(\nu m) m^\gamma <  \infty,  \label{eq:Phi_convergence3} \\ &\suminf \Phi_m'''(\nu m - \sqrt{2Km} )m^{\frac32} < \infty  \label{eq:Phi_convergence4}
\end{align}
\end{subequations}
where $\gamma\in(5/2, 3)$ is fixed. Here, we use \eqref{eq:Assumption_K} to ensure $\nu m - \sqrt{2Km} \geq 0$.

\begin{lemma}
Let   $\Phi_m(r) = r^{-\alpha}$ with $\alpha > \frac32$, then $\Phi_m$ satisfies the assumptions on $\Phi_m$. \label{lem:ralpha}
\end{lemma}
\begin{proof} This is a direct calculation.
\end{proof}

\begin{result}\label{THM}
    Under the assumptions on $\Phi_m$ and $K$, there exists a one-parameter family of solutions $(W_K,c_K)$ of \eqref{eq:wave_eq} with $\|W_K\|^2_2= 2K$.
\end{result}

\begin{remark}
    \begin{enumerate}
        \item The result can be expressed for the specific choice of the typical uniform lattices by using Lemma \ref{lem:ralpha}, i.e. for given $\nu>0$ and $\alpha>3/2$, there exists a family of  traveling waves of \eqref{eq:dynamic} for $\Phi_m(r)=r^{-\alpha}$
    with $\|W\|_2^2=2K$ for each  $K$ with $0<K<\nu^2/2$. The result also implies an existence result for lattices with repulsive finite-range forces, i.e. $\Phi_m \equiv 0$ for $m \geq m_0$ for some fixed $m_0$.
\item The high-energy limit for $K \to \nu^2/2$ is described in detail below for the potential $r^{-\alpha}$. The behavior is analogous to the standard FPUT case as analysed e.g. in \cite{HM15}. In  section 3, we establish that $W_K$ converges to an indicator function and that $c_K \to \infty$.
\item For $K \to 0$, it can be expected that the  traveling waves approach the long-wave solutions as given in  \cite{IP25} and \cite{AW25}, the rigorous limit in our variational setting is left out for reasons of brevity. A parallel, entirely variational analysis is done in \cite{Herrmann2020} for general related convolution operators for a long-wave limit approaching the KdV equation.

    \end{enumerate}
\end{remark}

\subsection{Further related results}

Early work on variational proofs for the existence of traveling waves are \cite{FW94} and \cite{FV99}. There is a comprehensive overview on traveling waves in FPUT  including models with next-nearest neighbors interactions by Vainchtein,
\cite{Vai22}. Finite-range interactions, i.e. $\Phi_m \equiv 0$ for $m \geq m_0$ for some fixed $m_0$ were studied via a long-wave KdV approach in \cite{HML16}, which was adapted to the infinite range case in \cite{IP25} and \cite{AW25} to get small amplitude, long-wave traveling waves.  We note that Ingimarson and Pego (\cite{IP25}) show the existence of traveling  waves for the setting of Lemma \ref{lem:ralpha} for the  range of parameters $4/3< \alpha<3$. They encounter different qualitative behavior for those with $\alpha \in (4/3,3/2)$ in that the Hamiltonian energy of the waves is not monotonically increasing in the wave speed, this is  in contrast to what can be expected for waves constructed by our variational approach for small $K$.

For finite-range interaction variational methods gave the existence of traveling waves of finite size in \cite{HM19a} and \cite{Pan24}.  Our approach is using a variational principle similar to  the existence proof for FPUT traveling waves in \cite{Her10}.  Similar adaptions of \cite{Her10} were made in \cite{HM19a} for peridynamical media (and finite-range FPUT) and in \cite{Herrmann2020} to a general convolution setting, see these papers also for discussions of other variational methods to construct solutions of FPUT.

\subsection*{Plan of paper} We introduce the variational principle in subsection \ref{subsec:2.1} on a suitable cone of unimodal functions. The existence of maximizers is shown in subsection \ref{sub:2.2} using a concentration compactness argument. The proof of the main theorem is given in subsection \ref{sub:2.3}. In section \ref{sec:HE} we discuss the asymptotic behavior for high-energy waves, i.e. for $K \to \nu^2/2$.

\section{Construction of Traveling Wave Solutions}

The proof of existence follows the variational approach  of \cite{Herrmann2020} where nonlinear, nonlocal eigenvalue problems of a form that include nearest-neighbor FPUT chains are analyzed, similar methods were adapted in \cite{HM19a} for peridynamical media which contained the more general FPUT case $\Phi_1=\Phi_2=\ldots=\Phi_M$ and $\Phi_j=0 $ for $j > M$. \\
While the overall structure of the argument follows the approach of \cite{Herrmann2020}, the present setting introduces substantial new analytical difficulties. In contrast to the finite-range interactions considered there, each particle in our model interacts with every other particle. As a consequence, several steps of the variational analysis require nontrivial modifications. In particular, it becomes necessary to establish the well-posedness of certain infinite series, such as the series $\mathcal{P}$ defined in the next subsection, a question that does not arise in \cite{Herrmann2020} or related finite-range models.
While some proofs of lemmas only require minor modifications, we provide all proofs to keep the exposition self-contained.

\subsection{Variational Setting}
\label{subsec:2.1}

 Following the variational method that is described in \cite{Herrmann2020}, we first define the potential and kinetic energy of $W$ by
\begin{align}
\Pcal(W) := \intR \suminf \Psi_m ( A_m W(s)) \, {\rm d}s, \qquad \Kcal(W) := \frac12 \intR \big( W(s) \big)^2 \, {\rm d}s,
\end{align}
respectively. Here, $\Psi_m :[0,\nu m) \to [0,\infty)$ is obtained by correcting the first two terms of the Taylor approximation of $\Phi_m$ around $\nu m$ and is given by
\begin{align}\label{def:psi}
\Psi_m(r) :=  \Phi_m(\nu m - r) - \Phi_m(\nu m) + \Phi_m'(\nu m)r.
\end{align}
We constrain our optimization to cones. Here, $\Ccal$ is  the convex cone given by all square integrable, even, unimodal and nonnegative functions, i.e.
\begin{align}
\Ccal := \overline{\{ W \in C_c^\infty(\Rset) \mid W(x) = W(-x) \geq 0 \quad \text{and} \quad \dot{W}(x) = -\dot{W}(-x) \leq 0 \quad \text{for all } x \geq 0 \}},
\end{align}
where the overline denote the $\mathsf{L}^2(\Rset)$ closure. Due  to convexity, $\Ccal$ is also closed under weak convergence in $\mathsf{L}^2(\Rset)$.
Then for every given $0 < K < \frac{\nu^2}{2}$ we define the cone
\begin{align}
\CK := \{W \in \Ccal \mid \Kcal(W) = K\}.
\end{align}

The general idea is to find a converging sequence $(W_n)_{n\in\Nset}$ of functions $W_n \in \CK$ that maximizes the potential energy such that its limit $W_{\infty}$ is a solution to the traveling wave problem. First, we prove some auxiliary results and ensure that the potential energy $\Pcal$ is indeed well-defined. We observe some further global bounds.

\begin{remark}
   The bound in \eqref{eq:Phi_convergence} implies with $\Phi_m' \leq 0$ and $\Phi_m'' \geq 0$ the global estimate
\begin{align}
\label{Phi'm}
    \suminf \Phi_m'(\nu m) m < \infty.
\end{align}
\end{remark}
The corrected potentials $\Psi_m$ as in \eqref{def:psi} have a minimum at $r=0$. We collect further properties in the next Lemma.
\begin{lemma}
\label{le:psi_aux}
For all $m \in \Nset$, the potential $\Psi_m$ satisfies
\begin{align}
\Psi_m(0) = \Psi_m'(0) = 0, \quad \Psi_m''(0) = \Phi_m''(\nu m) \geq 0
\end{align}
and
\begin{align}
\Psi_m(r) & \geq 0, \\
\Psi_m'(r) &= -\Phi_m'(\nu m - r) + \Phi'(\nu m) \geq 0, \\
\Psi_m''(r) &= \Phi_m''(\nu m - r) \geq 0
\end{align}
for all $r \in [0,\nu m)$.
\end{lemma}
\begin{proof}
These are direct consequences of the definition of $\Psi_m$ and the assumptions on $\Phi_m$ and its derivatives, e.g. a Taylor expanding of $\Phi_m$ around $\nu m$ reveals
\begin{align}
\Phi_m(\nu m -r) = \Phi_m(\nu m) - \Phi_m'(\nu m)r + \Phi_m''(\nu m - \theta r)\frac{r^2}{2!}
\end{align}
with $\theta \in (0,1)$, which implies
\begin{align}
\label{eq:Psi_m positive}
\Psi_m(r) = \Phi_m''(\nu m - \theta r)\frac{r^2}{2!} \geq 0.
\end{align}
\end{proof}
Next we state properties of the convolution operator applied to elements of the cone $\CK$.
\begin{lemma}
\label{le:AmW} Let $W \in \CK$ and  $m \in \Nset$,  then
$0 \leq A_m W(s) \leq \sqrt{2Km} $  for all $s \in \Rset$.
\end{lemma}
\begin{proof} The non-negativity of $A_m W(s)$ is trivial due to the definition of $A_m$ and $W$ being nonnegative. Young's  convolution inequality yields
\begin{align}
A_m W(s) \leq \|A_m W\|_{\infty} = \| \chi_{[-m/2,m/2]} \ast W \|_{\infty} \leq \| \chi_{[-m/2,m/2]} \|_2 \| W \|_{2} = \sqrt{m} \sqrt{2K}.
\end{align}
\end{proof}
Now we are in the position to show that $\Pcal(W)$ is finite for $W \in \CK$.
\begin{lemma}
\label{le: P_well_def}
The potential energy $\Pcal$ is well-defined for all $W \in \CK$.
\end{lemma}
\begin{proof} Lemma \ref{le:psi_aux} and \ref{le:AmW} imply for some $\theta \in (0,1)$
\begin{align}
\label{eq:Psi_m approx Phi''}
0 \leq  \Psi_m(A_mW(s)) =  \Phi_m(\nu m - \theta A_m W)\frac{(A_m W(s))^2}{2} \leq \Phi_m''\big(\nu m - \sqrt{2Km}\big) \frac{(A_mW(s))^2}{2},
\end{align}
using in the last step that $\Phi''$ is monotonically increasing. Note that the expressions $\Phi_m''(\nu m - \sqrt{2Km})$ are well-defined for all $m \in \Nset$ due to the condition $K < \frac{\nu^2}{2}$. With \mbox{$C_{m,K} := \frac12 \Phi_m''\big(\nu m - \sqrt{2Km}\big)$} and H\"older's inequality we then get
\begin{align}
\label{eq:int Psi_m}
\intR  \Psi_m(A_mW(s)) \, {\rm d}s &\leq C_{m,K} \intR ( A_mW(s) )^2 \,{\rm d}s 
= C_{m,K} \intR \Big[ \int\limits_{-\frac{m}{2}}^{\frac{m}{2}} W(s+\tau) \, {\rm d}\tau \Big]^2 \,{\rm d}s \nonumber \\
&\leq C_{m,K} \intR \Big[m \, \int\limits_{-\frac{m}{2}}^{\frac{m}{2}} W^2(s+\tau) \, {\rm d}\tau \Big] \,{\rm d}s  
= m \, C_{m,K} \int\limits_{-\frac{m}{2}}^{\frac{m}{2}} \intR W^2(s+\tau) \, {\rm d}s \,{\rm d}\tau \nonumber \\
&= m \, C_{m,K} \int\limits_{-\frac{m}{2}}^{\frac{m}{2}} 2K \,{\rm d}\tau 
= 2K m^2 C_{m,K}.
\end{align}
Substituting the results into the formula for $\Pcal$ and using Fubini's theorem we get
\begin{align}
\Pcal(W) &=  \intR \suminf \Psi_m ( A_m W(s)) \, {\rm d}s 
= \suminf \intR \Psi_m ( A_m W(s)) \, {\rm d}s  \nonumber\\
&\leq \suminf 2K m^2 C_{m,K} 
= \suminf K m^2 \Phi_m''(\nu m - \sqrt{2Km}) < \infty,
\end{align}
where the last step follows by assumption \eqref{eq:Phi_convergence2}.
\end{proof}

We will consider a maximizing sequence $(W_n)_{n\in\Nset}$. Such sequences can be constructed using an improvement operator
\begin{align}
\Tcal(W) := \mu(W)\partial\Pcal(W)
\end{align}
with
\begin{align}
\mu(W) := \frac{\|W\|_2}{\|\partial \Pcal(W)\|_2} \quad \text{and} \quad \partial \Pcal(W) = \suminf A_m \Psi_m'(A_m W).
\end{align}

We will not use an explicitly constructed sequence, but invariance properties of the cone $\CK$ will be relevant.

\begin{lemma}
The gradient $\partial\Pcal$ is well-defined for all $W \in \CK$.
\end{lemma}
\begin{proof}
    Let $V \in \mathsf{L}^2$ be fixed. Analogous to the computations in \eqref{eq:int Psi_m} follows
    \begin{align}
        \|A_m V\|_2 \leq m \| V \|_2.
    \end{align}
    In addition, a Taylor expansion of $\Phi_m'$ and Young's convolution inequality yield the estimate
    \begin{align}
        \|\Psi_m'&(A_m W) \|_2 = \Big( \intR \big( \Psi_m'(A_m W(s) ) \big)^2 \, {\rm d} s \Big)^{\frac12} \nonumber \\
        &= \Big( \intR \big( \Phi_m'(\nu m) - \Phi_m'(\nu m - A_m W(s) ) \big)^2 \, {\rm d} s \Big)^{\frac12} \nonumber \\
        &\leq  \Big( \intR \big( \Phi_m''(\nu m - \sqrt{2Km} ) (A_m W(s) )^2  {\rm d} s \Big)^{\frac12} \nonumber \\
        &= \Phi_m''(\nu m - \sqrt{2Km} ) \|A_m W\|_2 \leq \Phi_m''(\nu m- \sqrt{2Km}) \sqrt{2K} m.
    \end{align}
    In summary, using the Cauchy-Schwarz inequality we obtain
    \begin{align}
        \langle \partial &\Pcal(W), V \rangle = \suminf \langle A_m \Psi_m'(A_m W), V \rangle \\
        &= \suminf \langle \Psi_m'(A_m W), A_m V \rangle \leq \suminf \|\Psi_m'(A_m W) \|_2 \|A_m V\|_2 \nonumber \\
        &\leq \suminf \sqrt{2K} \|V\|_2 \Phi_m''(\nu m - \sqrt{2Km}) m^2 \leq C \|V\|_2
    \end{align}
    by \eqref{eq:Phi_convergence2}. Hence $\partial \Pcal(W) \in \mathsf{L}^2$ as required.
\end{proof}

\begin{lemma}
\label{le:T_aux} Let
$W \in \CK$, then $\Tcal(W) \in \CK$.
\end{lemma}
\begin{proof} Let $W \in \CK$. Using $\mu(W) \geq 0$ and $\Psi_m'(r) \geq 0$ for all $r \geq 0$ we get the nonnegativity of $\Tcal(W)$ via
\begin{align}
W \geq 0 &\Rightarrow A_m W \geq 0 \Rightarrow \Psi_m'(A_m W) \geq 0 \Rightarrow A_m \Psi_m'(A_m W) \geq 0 \nonumber \\
&\Rightarrow \suminf A_m \Psi_m'(A_m W) \geq 0 \Rightarrow \Tcal(W) \geq 0.
\end{align}
Furthermore, if $W$ is unimodal, then $A_m W$ and $\Psi_m'(A_m W)$ are also since $\Psi_m'$ is monotonically increasing. Hence, $A_m \Psi_m'(A_m W)$ is also unimodal and it follows that $\Tcal(W)$ is unimodal. That  $ \Tcal(W)$ is even if $W$ is even follows similarly. \\
We conclude the proof with the observation $\| \Tcal(W) \|_2 = \| W \|_2 = 2K$.
\end{proof}
The operator $\Tcal$ strictly increases the potential energy unless we have already a traveling wave solution.

\begin{lemma}
\label{le:T}
Let  $W \in \CK$, then $\Pcal(\Tcal(W)) \geq \Pcal(W)$. Moreover, equality holds if and only if $W = \Tcal(W)$, i.e. $W$ satisfies \eqref{eq:wave_eq} with $c = \mu(W)^{-\frac12}$.
\end{lemma}
\begin{proof} $\Psi_m''(r) \geq 0$ implies that $\Pcal$ is convex and hence,
\begin{align}\label{eqn:Pconvex}
\Pcal(V) - \Pcal(W) \geq \langle \partial \Pcal(W), V-W \rangle
\end{align}
holds for all $V,W \in \mathsf{L}^2(\Rset)$. This result and Lemma \ref{le:T_aux} yield
\begin{align}
\Pcal(\Tcal(W)) - \Pcal(W) &\geq \langle \partial \Pcal(W), \Tcal(W)-W \rangle 
= \frac{\langle \Tcal(W), \Tcal(W)-W \rangle }{\mu(W)}
\nonumber \\
&= \frac{\|\Tcal(W)\|_2^2 - 2 \langle\Tcal(W),W\rangle + \|W\|_2^2}{2\mu(W)} 
= \frac{\|T(W)-W\|_2^2}{2\mu(W)}.
\end{align}
It follows that $\Pcal(\Tcal(W)) \geq \Pcal(W)$ and equality holds if and only if $W = \Tcal(W)$. Finally, if $W = \Tcal(W)$, we note
\begin{align}
W = \Tcal(W) &= \mu(W)\partial \Pcal(W) 
= \mu(W) \suminf A_m\Psi_m'(A_m W)  \nonumber \\
&= \mu(W) \suminf A_m \big( -\Phi_m'(\nu m - A_m W) + \Phi_m'(\nu m) \big) \nonumber \\
&= \mu(W) \suminf A_m \big( -\Phi_m'(\nu m - A_m W) \big) + \mu(W) \suminf A_m \big( \Phi_m'(\nu m) \big),
\end{align}
providing a solution to \eqref{eq:wave_eq} with
\begin{align}
c = \mu(W)^{-\frac12}
\end{align}
and
\begin{align}
\label{eq:eta}
\eta = \suminf A_m \Phi_m'(\nu m)= \suminf m \Phi_m'(\nu m).
\end{align}
Note that the  series in \eqref{eq:eta} converges by \eqref{Phi'm}.
\end{proof}

\subsection{Existence of Maximizers} \label{sub:2.2}

It is our goal to show that every maximizing sequence for $\Pcal$ in $\CK$ admits a strongly convergent subsequence. This will complete  the existence proof of a maximizer $W \in \CK$, which is then a solution  of \eqref{eq:wave_eq}. 
For this purpose, we define -- based on the quantities in \cite{Herrmann2020} -- the modified quantities
\begin{align}
P(K) := \sup\limits_{W \in \CK} \Pcal(W), \qquad Q(K) := \sup\limits_{W \in \CK} \Qcal(W)
\end{align}
with
\begin{align}
\Qcal(W) := \frac{1}{2} \intR \suminf \Phi_m''(\nu m) ( A_m W(s) )^2 \,{\rm d}s.
\end{align}
Note that $\Qcal$ is the quadratic term of $\Pcal$ when Taylor expanding $\Psi_m$ around $0$ in the definition of $\Pcal$ and recalling $\Psi_m''(0) = \Phi_m''(\nu m)$. The subsequent Lemma quantifies $Q(K)$ before we prove in Lemma \ref{le:P(K)-Q(K)} that $P(K)$ is strictly greater than $Q(K)$, implying that the super-quadratic terms in $\Pcal(W)$ play a significant role.

\begin{lemma}
\label{le:Q(K)}
For all $K>0$,  $Q(K) = \suminf \Phi_m''(\nu m) K m^2$ holds.
\end{lemma}
\begin{proof}
Using H\"older's inequality and Fubini's theorem in the same way as in the proof of Lemma \ref{le: P_well_def}, we obtain an upper bound
\begin{align}
Q(K) \leq \suminf \Phi_m''(\nu m) K m^2.\label{eqn:Qkup}
\end{align}
Now, we consider a family of test functions to get matching lower bounds.  For a parameter $L \in \Nset$ such that $\sqrt{L} \in \Nset$, we define  $W_L \in \CK$  by
\begin{align}
\label{eq:W_L}
W_L(x) := \sqrt{\frac{K}{L}}\, \chi_{[-L,L]}(x)
\end{align}
and note that for all $m \leq 2L$ (and in particular $m \leq \sqrt{L}$) we have
\begin{align}
\label{eq:AmWL}
A_m W_L(\xi) = \begin{cases}
\sqrt{\frac{K}{L}}\,m \, &\text{if}\,\, 0 \leq |\xi| \leq L - \frac{m}{2}, \smallskip \\
\sqrt{\frac{K}{L}}\, (L - (\xi-\frac{m}{2})) \, &\text{if}\,\, L - \frac{m}{2} < |\xi| \leq L + \frac{m}{2}, \smallskip \\
0 \, &\text{if}\,\, |\xi| > L+\frac{m}{2}.
\end{cases}
\end{align}
With \eqref{eq:AmWL} it follows
\begin{align}
\label{eq:Q(W_L)}
\Qcal(W_L) &= \frac12 \suminf \Phi_m''(\nu m) \intR (A_m W_L(s) )^2 \,{\rm d}s \nonumber \geq \frac12 \sum_{m=1}^{\sqrt{L}}  \Phi_m''(\nu m) \int\limits_{-(L-\frac{m}{2})}^{L-\frac{m}{2}} (A_m W_L(s) )^2 \,{\rm d}s \nonumber \\
&= \frac12 \sum_{m=1}^{\sqrt{L}}  \Phi_m''(\nu m) \, \frac{K}{L} m^2 (2L-m) \nonumber = \sum_{m=1}^{\sqrt{L}} \Phi_m''(\nu m)  K  m^2 \Big( 1 - \frac{m}{2L} \Big) \nonumber \\
&\geq \sum_{m=1}^{\sqrt{L}} \Phi_m''(\nu m)  K  m^2 \Big( 1 - \frac{1}{2\sqrt{L}} \Big).
\end{align}
Note that the  integral occurring in the third term is positive since all $m \in \{1,...,\sqrt{L}\}$ satisfy $L-\frac{m}{2}  > 0$.
We obtain that
\begin{align}
\label{eq:sum0}
0\leq \lim\limits_{L \to \infty} \sum\limits_{m=1}^{\sqrt{L}} \Phi_m''(\nu m) K m^2 \,  \frac{1}{2\sqrt{L}} \leq \lim\limits_{L \to \infty}\frac{K}{2\sqrt{L}}   \sum\limits_{m=1}^{\infty} \Phi_m''(\nu m)  m^\gamma  = 0
\end{align}
holds by  \eqref{eq:Phi_convergence3} with $\gamma>5/2$. Finally, using \eqref{eq:Q(W_L)} and \eqref{eq:sum0}, we derive the corresponding lower bounds to \eqref{eqn:Qkup}
\begin{align}
Q(K) \geq \lim\limits_{L \to \infty} \Qcal(W_L) \geq \suminf  \Phi_m''(\nu m) K m^2,
\end{align}
which completes the proof.
\end{proof}

\begin{lemma}
\label{le:P(K)-Q(K)}
For all $K$ such that $0 < K< \frac{\nu^2}{2}$,  $P(K) > Q(K)$ holds.
\end{lemma}
\begin{proof}
As in the proof of Lemma \ref{le:Q(K)}, we consider the family of test functions $W_L$ defined in \eqref{eq:W_L} with parameter $L \in \Nset$.

\textit{Estimate for $\Pcal(W_L) - \Qcal(W_L)$:} First we observe that the Taylor expansion
\begin{align}
\Phi_m(\nu m - r) = \Phi_m(\nu m) - \Phi_m'( \nu m)r + \Phi_m''(\nu m)\frac{r^2}{2} - \Phi_m'''(\nu m - \theta r)\frac{r^3}{3!}
\end{align}
with some $\theta \in [0,1]$ implies
\begin{align}
\label{eq:estimate_cubic_Phi1}
\Psi_m(A_mW(s)) - \Phi_m''(\nu m) \frac{(A_mW(s))^2}{2} = -\Phi_m'''(\nu m - \theta A_m W(s) ) \frac{(A_m W(s))^3}{3!}.
\end{align}
Using Lemma \ref{le:AmW} and the fact that $\Phi_m'''$ is nonpositive and monotonically increasing, we obtain
\begin{align}
\label{eq:estimate_cubic_Phi2}
-\Phi_m'''(\nu m - \theta A_m W(s) ) \frac{(A_m W(s))^3}{3!} \geq C_m (A_m W(s))^3
\end{align}
with the constant $C_m := -\Phi_m'''(\nu m)/3! \geq 0$.

Combining \eqref{eq:AmWL}, \eqref{eq:estimate_cubic_Phi1} and \eqref{eq:estimate_cubic_Phi2} we then obtain
\begin{align}
\label{eq:P(WL)-Q(WL)aux}
\Pcal(W_L) - \Qcal(W_L) &= \intR \suminf \Psi_m(A_m W_L(s)) - \Phi_m''(\nu m) (A_m W_L(s) )^2 \, {\rm d}s  \nonumber\\
&\geq \intR \suminf C_m (A_m W_L(s))^3 
\geq C_1 \int\limits_{-(L-\frac12)}^{L-\frac12} (A_1 W_L(s))^3 \, {\rm d}s  \nonumber \\
&= \frac{C_1 K^{\frac32}}{L^\frac32} (2L-1 ) 
= c_1 (2L^{-\frac12} - L^{-\frac32})
\end{align}
for the constant $c_1 := C_1K^{\frac32} > 0$. Here, we use the strict negativity of $\Phi'''_1$ in \eqref{eqn:assPhi}.

\textit{Estimate for $Q(K) - \Qcal(W_L)$:}
Analogously to the computations in \eqref{eq:Q(W_L)} we estimate
\begin{align}
\label{eq:Qcal}
    \Qcal(W_L) \geq \sum_{m=1}^{L} \Phi_m''(\nu m)  K  m^2 \Big( 1 - \frac{m}{2L} \Big).
\end{align}
With Lemma \ref{le:Q(K)} and \eqref{eq:Qcal} we observe
\begin{align}
\label{eq:Q(K)-Q(W_L)}
Q(K) - \Qcal(W_L) &\leq \suminf \Phi_m''(\nu m) K m^2 - \sum\limits_{m=1}^{L} \Phi_m''(\nu m) K m^2 \Big( 1- \frac{m}{2L} \Big)  \nonumber\\
&= \sum\limits_{m=L+1}^{\infty} \Phi_m''(\nu m) K m^2 + \sum\limits_{m=1}^{L} \Phi_m''(\nu m) K \frac{m^3}{2L}.
\end{align}
Using the global bound in \eqref{eq:Phi_convergence3}, we get for the first sum of \eqref{eq:Q(K)-Q(W_L)} that
\begin{align}
    &\sum\limits_{m=L+1}^{\infty} \Phi_m''(\nu m) K m^2 = K \sum\limits_{m=L+1}^{\infty} \Phi_m''(\nu m) m^\gamma m^{2-\gamma  } \nonumber \\
    &\leq K \sum\limits_{m = L + 1}^\infty \Phi_m''(\nu m) m^\gamma  \big(L+1\big)^{\gamma -2} \nonumber
    = \big(L+1\big)^{-(\gamma -2)} K \sum\limits_{m = L + 1}^\infty \Phi_m''(\nu m) m^\gamma \nonumber \\
    &\leq c_2 L^{-(\gamma -2)}
\end{align}
holds with $c_2 \geq 0$. \\
Analogously, the second sum in \eqref{eq:Q(K)-Q(W_L)} can be estimated by
\begin{align}
    \sum\limits_{m=1}^{L} \Phi_m''(\nu m) K \frac{m^3}{2L} \leq c_3 \frac{L^{3-\gamma}}{L}  = c_3 L^{-(\gamma -2)}, \quad c_3 \geq 0.
\end{align}
\textit{Combining all results:} We now have
\begin{align}
P(K) \geq \Pcal(W_L) &= \Qcal(W_L) + \Pcal(W_L) - \Qcal(W_L) \nonumber  \\
&\geq Q(K) - c_2 L^{{-(\gamma -2)}} - c_3 L^{{-(\gamma -2)}} + c_1(2L^{-\frac12} - L^{-\frac32} ) .
\end{align}
As $-(\gamma -2)< -\frac12$ by assumption, we obtain the inequality  $P(K) > Q(K)$ by choosing $L$ finite but sufficiently large.
\end{proof}

We are now in the position to prove the main technical step.

\begin{proposition}
\label{le:strong convergence}
Any sequence $(W_n)_{n\in\Nset} \subseteq \CK$ that satisfies $\lim\limits_{n\to\infty} \Pcal(W_n) = P(K)$ admits a strongly convergent subsequence in $\mathsf{L}^2(\Rset)$.
\end{proposition}
\begin{proof}
\textit{Preliminaries:} Since $(W_n)_n$ is a bounded sequence in $\mathsf{L}^2$, there exists a (not relabeled) subsequence with
\begin{align}
\label{eq:weak}
W_n \rightharpoonup  W_{\infty} \quad  \text{ for } n \to \infty \text{ weakly in} \, \mathsf{L}^2(\Rset).
\end{align}
Note that $\Ccal$ is convex and closed and thus, $W_{\infty} \in \Ccal$. It is our goal to show
\begin{align}
\|W_{\infty}\|_2^2 \geq 2K,
\end{align}
since this implies $\|W_{\infty}\|_2^2 = 2K$ due to
\begin{align}
\|W_{\infty}\|_2^2 \leq \liminf\limits_{n\to\infty} \|W_n\|_2^2 = 2K.
\end{align}
The strong convergence then follows directly from the weak convergence.
\\
\textit{Defining new quantities:} For given cut-off parameters $L,M \in \Nset$ we define the functions
\begin{align}
\widetilde{W}_n(s) := W_n(s) \chi_{[-L,L]}(s) \qquad \text{and} \qquad \overline{W}_n(s) := W_n(s) - \widetilde{W}_n(s)
\end{align}
and the quantities
\begin{align}
\widetilde{\Pcal}(W) := \sum\limits_{m=1}^{M} \intR \Psi_m(A_m W(s)) \,{\rm d}s \qquad \text{and} \qquad \widetilde{\Qcal}(W) := \sum\limits_{m=1}^{M} \intR \frac12 \Phi_m''(\nu m) (A_m W(s))^2 \,{\rm d}s.
\end{align}
We note that the identity
\begin{align}
\label{eq:Wt + Wo}
    \|\Wt_n\|_2^2 + \|\Wo_n\|_2^2 = \|W_n\|_2^2 = 2K
\end{align}
holds for all $n \in \Nset$. \\
\textit{Approximations:} Let $ \epsilon >0$ be given. Our first observation is that by choosing $M$ large enough we can guarantee
\begin{align}
0 \leq | \Pcal(W_n) - \widetilde{\Pcal}(W_n)| \leq \sum_{m=M+1}^\infty K m^2 \Phi_m''(\nu m - \sqrt{2Km}) \leq \epsilon
\end{align}
uniformly in $\Ccal_K$ due to the estimates in the proof of Lemma \ref{le: P_well_def}.
Secondly, we note that $A_m \widetilde{W}_n$ is supported in $[-L-\frac{m}{2}, L+\frac{m}{2}]$, while $A_m \overline{W}_n$ is supported in $\Rset \setminus [-L+\frac{m}{2}, L-\frac{m}{2}]$. With that we derive
\begin{align}
\label{eq:estimate P_tilde}
&|\Pt(W_n) - \Pt(\Wt_n)-\Pt(\Wo_n)| \nonumber\\ &= \Big| \sumM \intR \Psi_m(A_m W_n(s)) - \Psi_m(A_m \Wt_n(s)) - \Psi_m(A_m \Wo_n(s)) \, {\rm d}s \Big| \nonumber \\
&= \sumM \Big( \int\limits_{-L-\frac{m}{2}}^{-L+\frac{m}{2}} \Psi_m(A_m W_n(s)) \, {\rm d}s + \int\limits_{L-\frac{m}{2}}^{L+\frac{m}{2}} \Psi_m(A_m W_n(s)) \, {\rm d}s \Big) \nonumber \\
&= 2 \sumM \int\limits_{L-\frac{m}{2}}^{L+\frac{m}{2}} \Psi_m(A_m W_n(s)) \, {\rm d}s \leq 2 \sumM \int\limits_{L-\frac{M}{2}}^{L+\frac{M}{2}} \Psi_m(A_m W_n(s)) \, {\rm d}s.
\end{align}
Furthermore, we note that by the unimodality of $A_m W_n$
\begin{align}
\label{eq:A_mW_n1}
    \|A_m W_n \|_2^2 &= \intR \big(A_m W_n (\sigma) \big)^2 \, {\rm d}\sigma \geq  \int_{-|s|}^{|s|} \big(A_m W_n (\sigma) \big)^2 \, {\rm d}\sigma \nonumber \\
    &\geq  \int_{-|s|}^{|s|} \big(A_m W_n (s) \big)^2 \, {\rm d}\sigma = 2|s| \big(A_m W_n(s) \big)^2
\end{align}
holds for all $s \in \Rset$. Hence we have with Young's convolution inequality
\begin{align}
\label{eq:A_mW_n}
0 \leq \big(A_m W_n(s)\big)^2 \leq \frac{\|A_m W_n\|_2^2}{2|s|} \leq \frac{\|A_m\|_1^2 \|W_n\|_2^2}{2|s|} =  \frac{Km^2}{|s|}.
\end{align}
Using \eqref{eq:Psi_m approx Phi''}, \eqref{eq:estimate P_tilde} and \eqref{eq:A_mW_n} we find
\begin{align}
    |\Pt(W_n)& - \Pt(\Wt_n)-\Pt(\Wo_n)| \leq 2 \sumM \int\limits_{L-\frac{M}{2}}^{L+\frac{M}{2}} \frac{\Phi_m''(\nu m - \sqrt{2Km})}{2} \big(A_m W_n(s) \big)^2 \, {\rm d}s \nonumber \\
    &\leq \sumM \Phi_m''(\nu m - \sqrt{2Km})m^2K \int\limits_{L-\frac{M}{2}}^{L+\frac{M}{2}} \frac{1}{s} \, {\rm d}s = C \ln\Big(\frac{L+\frac{M}{2}}{L-\frac{M}{2}}\Big) \leq \epsilon
\end{align}
for any $M$ fixed and $L$ large enough. \\
By Taylor expanding $\Phi_m$ we obtain
\begin{align}
     \Phi_m(\nu m - r) = \Phi_m(\nu m) - \Phi'(\nu m) r + \frac{\Phi_m''(\nu m)}{2}r^2 - \frac{\Phi_m'''(\nu m - \theta r)}{3!}r^3
 \end{align}
for some suitable $\theta \in (0,1)$, which implies with Lemma \ref{le:AmW} and $\Phi'''_m \geq 0$ respective $\Phi_m^{(4)} \leq 0$
\begin{align}
\label{eq:TaylorPhi}
     \Psi_m(\nu m - A_m W(s)) - \frac{\Phi_m''(\nu m)}{2}(A_m W(s))^2 \leq - \frac{\Phi_m'''(\nu m - \sqrt{2Km})}{3!}r^3
 \end{align}
With \eqref{eq:TaylorPhi} we then estimate
\begin{align}
\label{eq:tilde W_o}
    | \widetilde{\Pcal}(\Wo_n) - \widetilde{\Qcal}(\Wo_n) | &=\Big| \sumM \intR \Psi_m(A_m\Wo_n(s) ) - \frac{\Phi_m''(\nu m)}{2} (A_m \Wo_n)(s))^2 \, {\rm d} s \Big| \nonumber \\
    &\leq\Big| \sumM \intR -\frac{\Phi_m'''(\nu m - \sqrt{2Km}  )}{3!} (A_m\Wo_n(s))^3 \, {\rm d} s \Big| \nonumber \\ &= 2 \Big| \sumM \int\limits_{L-\frac{m}{2}}^{\infty} \frac{\Phi_m'''(\nu m - \sqrt{2Km})}{3!} (A_m W_n(s))^3 \, {\rm d} s \Big| \nonumber \\
    &\leq 2 \Big| \sumM \frac{\Phi_m'''(\nu m - \sqrt{2Km})}{3!} (2Km)^{\frac32} \int\limits_{L-\frac{m}{2}}^{\infty} \frac{1}{s^{\frac32}} \, {\rm d} s \Big|
\end{align}
using the fact that $A_m\Wo_n$ is even and supported in $\Rset \setminus[-L+\frac{m}{2}, L-\frac{m}{2}]$. By assumption \eqref{eq:Phi_convergence4}
and the decay of $W_n(s)$ for $|s| \to \infty$  as in \eqref{eq:A_mW_n}, we can estimate
\begin{align}
    | \widetilde{\Pcal}(\Wo_n) - \widetilde{\Qcal}(\Wo_n) | \leq \epsilon
\end{align}
uniformly in $n$ by choosing $L$ in \eqref{eq:tilde W_o} sufficiently large. In summary, we obtain the estimate
\begin{align}
\label{eq:Approx4eps}
    &| \Pcal(W_n) - \Pt(\Wt_n) - \Qt(\Wo_n) | \nonumber\\
    &\leq |\Pcal(W_n) - \Pt(W_n) | + | \Pt(W_n) - \Pt(\Wt_n) - \Pt(\Wo_n) | + |\Pt(\Wo_n) - \Qt(\Wo_n) | \leq 3 \epsilon.
\end{align}
In particular, this implies
\begin{align}
\label{eq:Approx4eps2}
    \Pcal(W_n) \leq \Pt(\Wt_n) + \Qt(\Wo_n) + 3\epsilon.
\end{align}
Moreover, choosing once again $L$ large enough, we estimate
\begin{align}
    | \Qt(\Wo_{\infty}) | \leq \epsilon.
\end{align}
The weak convergence \eqref{eq:weak} $W_n \rightharpoonup W_{\infty}$ implies the pointwise convergence
\begin{align}
    A_m \Wt_n \xrightarrow{n\to\infty} A_m \Wt_{\infty}.
\end{align}
We already noted that $A_m\Wt_n$ is supported in $[-L-\frac{m}{2}, L+\frac{m}{2}]$ and bounded by $\sqrt{2Km}$, thus  strong convergence
\begin{align}
\label{eq:strong A_m Wt}
    A_m \Wt_n \xrightarrow{n \to \infty} A_m \Wt_{\infty}
\end{align}
follows in $\mathsf{L}^2$ by Lebesgue's dominated convergence theorem. With \eqref{eq:strong A_m Wt} we find
\begin{align}
\label{eq:Pt(Wt_n)}
    \Pt(\Wt_n) \leq \Pt(\Wt_{\infty}) + \epsilon.
\end{align}
In addition, since $(W_n)_{n\in \Nset}$ is a maximizing sequence, by choosing $n$ sufficiently large, we can ensure
\begin{align}
\label{eq:Pcal(W_n)}
    \Pcal(W_n) \geq P(K) - \epsilon.
\end{align}
With \eqref{eq:Approx4eps2}, \eqref{eq:Pt(Wt_n)} and \eqref{eq:Pcal(W_n)} we then estimate
\begin{align}
\label{eq:Approx P(K)}
    P(K) &\leq \Pcal(W_n) + \epsilon \leq \Pt(\Wt_n) + \Qt(\Wo_n) + 4\epsilon \nonumber \\
    &\leq \Pt(\Wt_{\infty}) + \Qt(\Wo_n) + 5 \epsilon \leq \Pcal(\Wt_{\infty})+\Qcal(\Wo_n)+5\epsilon.
\end{align}
Using that $\Pcal$ is super-quadratic, we find
\begin{align}
\label{eq:Approx Pcal(Wt)}
    \Pcal(\Wt_{\infty}) \leq \Pcal \Big( \frac{\sqrt{2K}}{\|\Wt_{\infty}\|_2} \Wt_{\infty} \Big) \frac{\|\Wt_{\infty}\|_2^2}{2K} \leq P(K) \frac{\|\Wt_{\infty}\|_2^2}{2K}.
\end{align}
Moreover,
\begin{align}
    \Qcal(\lambda W) = \suminf \intR \frac{\Phi_m''(\nu m)}{2} \big( A_m \lambda W (s) \big)^2 = \lambda^2 \Qcal(W)
\end{align}
holds by definition and thus,
\begin{align}
\label{eq:Approx Qcal(Wo)}
    \Qcal(\Wo_n) = \Qcal \Big( \frac{\sqrt{2K}}{\|\Wo_n\|_2} \Wo_n \Big) \frac{\|\Wo_n\|_2^2}{2K}  \leq  Q(K) \frac{\|\Wo_n\|_2^2}{2K}.
\end{align}
With \eqref{eq:Approx P(K)}, \eqref{eq:Approx Pcal(Wt)} and \eqref{eq:Approx Qcal(Wo)} it follows
\begin{align}
\label{eq:P(K) 6 eps}
    P(K) \leq P(K) \frac{\|\Wt_{\infty}\|_2^2}{2K} + Q(K) \frac{\|\Wo_n\|_2^2}{2K} + 5\epsilon.
\end{align}
Substituting $Q(K) = P(K) - (P(K) - Q(K))$ and rearranging the terms in the last inequality implies
\begin{align}
\label{eq:P(K) 6 eps 2}
    P(K) + \frac{\|\Wo_n\|_2^2}{2K} (P(K) - Q(K) ) \leq \frac{\|\Wt_{\infty}\|_2^2 + \|\Wo_n\|_2^2}{2K} P(K) + 6\epsilon.
\end{align}
The weak convergence $W_n \rightharpoonup W_{\infty}$ implies the weak convergence $\Wt_n \rightharpoonup \Wt_{\infty}$. This yields
\begin{align}
   \liminf\limits_{n\to\infty} \|\Wt_n\|_2 \geq \|\Wt_{\infty}\|_2
\end{align}
and hence for any $\tilde{\epsilon}>0$
\begin{align}
    \|\Wt_n\|_2 \geq \|\Wt_{\infty}\|_2 (1-\tilde{\epsilon})
\end{align}
for sufficiently large $n \in \Nset$. Together with \eqref{eq:Wt + Wo} it follows
\begin{align}
\label{eq:1+eps}
    \|\Wt_{\infty}\|_2^2 + \|\Wo_n\|_2^2 \leq \frac{\|\Wt_n\|_2^2}{(1-\tilde{\epsilon})^2} + \|\Wo_n\|_2^2 = \frac{2K - (1-(1-\tilde{\epsilon})^2) \|\Wo_n\|_2^2}{(1-\tilde{\epsilon})^2} \leq 2K (1+\epsilon)
\end{align}
for $n$ large enough and with the substitution $\epsilon := \frac{1-(1-\tilde{\epsilon})^2}{(1-\tilde{\epsilon})^2}$. Combining the results from \eqref{eq:P(K) 6 eps 2} and \eqref{eq:1+eps} we find
\begin{align}
    \|\Wo_n\|_2^2 \leq \frac{2K ( P(K) + 5)}{P(K)-Q(K)}\epsilon = C \epsilon.
\end{align}
Inserting this into \eqref{eq:P(K) 6 eps} we find the approximation
\begin{align}
    P(K) \leq P(K) \frac{\|\Wt_{\infty}\|_2^2}{2K} + \tilde{C} \epsilon
\end{align}
which implies
\begin{align}
    \|\Wt_{\infty}\|_2^2 \geq 2K
\end{align}
by choosing $\epsilon > 0$ sufficiently small. We conclude the proof with the observation
\begin{align}
    \|W_{\infty}\|_2^2 \geq \|\Wt_{\infty}\|_2^2.
\end{align}
\end{proof}

\subsection{Proof of Main Theorem}
\label{sub:2.3}

We are now in the position to complete the proof of our main result.

\begin{proof}[Proof of Theorem \ref{THM}]
Proposition \ref{le:strong convergence} implies that there exists a maximizer $W \in \Ccal_K$ of $\Pcal$ that can be constructed as the limit point of a maximizing sequence. According to Lemma \ref{le:T}, this maximizer is a solution to \eqref{eq:wave_eq}.
\end{proof}

We end this section by stating some properties for the wave speed of such solutions.

\begin{lemma}
\label{eq:estimate c}
Let $(W_K, c_K)$ be a solution provided by Theorem \ref{THM}. Then
\begin{align}
     c_K = \mu(W_K)^{-\frac12} \quad \text{and} \quad
  2  c_K^2 \geq \frac{\Pcal(W_K)}{\Kcal(W_K)} = \frac{P(K)}{K}
\end{align}
hold.
\end{lemma}

\begin{proof} Using that $W_K$ is a maximizer, the first statement follows directly from Lemma \ref{le:T}. To show the other one, we test the equation
\begin{align}
    W_K =
\Tcal(W_K) = \mu(W_K)\partial\Pcal(W_K)
\end{align}
with $W_K$. This yields
\begin{align}
    c_K^2= \frac{1}{\mu(W_K)}= \frac{\langle \partial\Pcal(W_K), W_K \rangle}{\langle W_K, W_K \rangle} \geq
    \frac{\Pcal(W_K)}{2\Kcal(W_K)} = \frac{P(K)}{2K}
\end{align}
after evaluating \eqref{eqn:Pconvex} with $V=0$ and $W=W_K$.
\end{proof}

\section{High-Energy Limit} \label{sec:HE}

In this section we study the case where the kinetic energy
\begin{align}
\Kcal(W) = \frac12 \intR \big( W(s) \big) ^2 \, {\rm d}s
\end{align}
of the traveling wave profile $W$ converges to the high-energy limit $K_{\text{max}} := \nu^2/2$. A very detailed analysis of these  asymptotics for nearest neighbor FPUT is given in \cite{HM15}. The high-energy limit is dominated by the nearest neighbor interaction in our setting too, hence we expect that we could adapt the detailed analysis to the solution constructed in Theorem \ref{THM}.  A simpler analysis can still describe the limit behavior.  For that, we again adapt the ideas of \cite{Herrmann2020}. \\
For simplifications, we restrict our considerations to the special case
\begin{align}
\Phi_m(r) = r^{-\alpha} \label{eqn:PhimhiE}
\end{align}
for fixed $\alpha > \frac32$ and all $m \in \Nset$. Indeed, we only require the specific  singularity for $m=1$. The simplification implies
\begin{align}
\label{Psi_1}
\Psi_m(r) = (\nu m - r)^{-\alpha} - (\nu m)^{-\alpha} - \alpha (\nu m)^{-\alpha-1} r.
\end{align}
To start, let $0 < \delta < 1$ be a fixed value and consider a maximizing function $W_{\delta} \in \Ccal$ with
\begin{align}
\label{high energies}
\Kcal(W_{\delta}) = \big(1-\delta\big) \frac{\nu^2}{2}, \qquad \Pcal(W_{\delta}) = P\Big(\big(1-\delta\big)\frac{\nu^2}{2}\Big).
\end{align}
The first equation in \eqref{high energies} ensures that the kinetic energy $\Kcal(W_{\delta})$ tends to $K_{\text{max}}$ for $\delta \to 0$. The second equation means that we choose $W_{\delta}$ to be a maximizer  as we proved in the previous section, which might not be unique. 
It is our goal to analyze the limit $W_{\delta}$ for $\delta \to 0$. For this, we define the function
\begin{align}\label{eqn:W0}
W_0(x) := \nu \chi_{[-1/2, 1/2]}(x) = \begin{cases}
\nu, & x \in [-\frac12, \frac12], \\
0, & x \notin [-\frac12, \frac12].
\end{cases}
\end{align}

We will show that $W_0$ is indeed the correct candidate for the limit of $W_{\delta}$. As a first indication that this could be true, note the following lemma.

\begin{lemma}
The function $W_0$ as in \eqref{eqn:W0} satisfies $\Kcal(W_0) = K_{\text{max}}$.
\end{lemma}
\begin{proof}
This is a simple direct computation.
\end{proof}

Before formulating the convergence result, we also define the auxiliary variables
\begin{align}
U_{\delta, m} :=   A_m W_{\delta}  , \qquad \epsilon_{\delta, m} := \nu m - U_{\delta, m}(0)
\end{align}
for all $m \in \Nset$. The functions $U_{\delta, m}$ are unimodal by construction and thus the parameters $\epsilon_{\delta, m}$ denote the maximal distances between $U_{\delta, m}$ and the arguments of the singularities of $\Psi_m$, namely $\nu m$. \\

What follows is the main result of this section.

\begin{theorem}
We have $W_{\delta} \xrightarrow{\delta \to 0} W_0$ strongly in $\mathsf{L}^2(\Rset)$ and $c_{\delta}^2 \xrightarrow{\delta \to 0} \infty$. Here, $c_{\delta}$ denotes the wave speed of $W_{\delta}$.
\end{theorem}

\begin{proof}
\underline{Auxiliary results:} At first, we note that by Lemma
\ref{le:AmW}
\begin{align}
\label{eqn:epsm}
\epsilon_{\delta, m}\geq \nu(m -\sqrt{m}),
\end{align}
which is uniformly bounded away from $0$ for $m\geq 2$.

We want to show the convergence $\epsilon_{\delta, 1} \xrightarrow{\delta \to 0} 0$. For that we first notice that
\begin{align}
\label{kinetic of testfunction}
\Kcal\big(\sqrt{1-\delta}\, W_0\big) = \big(1-\delta\big) \frac{\nu^2}{2} = \Kcal(W_{\delta}) =: \Kcal_{\delta}
\end{align}
holds by construction and
\begin{align}
\label{computing A_1 W_0}
A_1 W_0 (s) = \int\limits_{-\frac12}^{\frac12} \nu \, \chi_{[-1/2, 1/2]} (s + \sigma) \, {\rm d} \sigma = \begin{cases}
(s+1)\nu, & -1 \leq s \leq 0, \\
(1-s)\nu, & 0 < s \leq 1, \\
0, & s \notin [-1, 1]
\end{cases}
\end{align}
is a tent-shaped function with support in $[-1, 1]$. \\
Using that $W_{\delta}$ is a maximizer of the potential energy along all functions with a kinetic energy of $K_{\delta}$, \eqref{kinetic of testfunction} and \eqref{computing A_1 W_0}, it follows
\begin{align}
\label{P_infinite}
\Pcal(W_{\delta}) &\geq \Pcal(\sqrt{1-\delta}\, W_0) \geq \int_1^{\infty} \Psi_1 \big (A_1 \sqrt{1-\delta} \, W_0 (s) \big) \, {\rm d}s \\ &\xrightarrow{\delta \to 0} \int_1^{\infty}  \Psi_1 \big (A_1 \, W_0 (s) \big) \, {\rm d}s \nonumber
\geq \int_1^{\infty}  \big (\nu - A_1 W_0(s) \big)^{-\alpha} \, {\rm d}s = \int_1^{\infty} \nu^{-\alpha} = \infty.
\end{align}
Suppose for the moment that $\epsilon_{\delta, 1}$ does not converge to zero. 
Then there exists a constant $\epsilon_{0,1} \in (0,\nu)$ such that $\epsilon_{\delta, 1} \geq \epsilon_{0,1}$ applies to a fixed sequence $(\epsilon_{\delta, 1})_{\delta}$ for $\delta \to 0$. With the definition and unimodality of $U_{\delta, 1}$ and the definition of $\epsilon_{\delta, 1}$ it follows that
\begin{align}
    \| A_1 W_{\delta}\|_{\infty} = A_1 W_{\delta}(0) = U_{\delta, 1}(0) = \nu - \epsilon_{\delta,1} \leq \nu - \epsilon_{0,1}
\end{align}
holds. This implies a strictly positive distance between $A_1 W_{\delta}$ and the argument $\nu$ of the singularity of $\Psi_1$ for $\delta \to 0$. Due to \eqref{eqn:epsm}, this distance result between $A_m W_{\delta}$ and its corresponding singularity $\nu m$ holds for all $m \in \Nset$ and convergence for $\Phi_m$ as in \eqref{eqn:PhimhiE} follows by inspection. With the arguments in the proof of Lemma \ref{le: P_well_def}, it follows that the limit $\lim\limits_{\delta \to 0} \Pcal(W_{\delta})$ is finite. This is a contradiction to \eqref{P_infinite} and therefore we conclude $\epsilon_{\delta, 1} \xrightarrow{\delta \to 0} 0$. 
\\
\underline{Proof of the statement:} With
\begin{align}
U_{\delta, 1}(0) = (A_1 W_{\delta})(0) = \int\limits_{-\frac{1}{2}}^{\frac{1}{2}} W_{\delta}(x) \, {\rm d}x
= \intR \chi_{[-1/2, 1/2]}(x) W_{\delta}(x) \, {\rm d}x = \langle W_{\delta}, \chi_{[-1/2, 1/2]} \rangle
\end{align}
we observe
\begin{align}
\langle W_{\delta}, W_0 \rangle = \nu  \langle W_{\delta}, \chi_{[-1/2, 1/2]} \rangle = \nu U_{\delta, 1}(0).
\end{align}
With $\epsilon_{\delta,1} \to 0$ it follows
\begin{align}
\|W_{\delta} - W_0 \|_2^2 &= \|W_{\delta}\|_2^2 + \|W_0\|_2^2 - 2 \langle W_{\delta}, W_0 \rangle = 2 \Kcal_{\delta} + 2 K_{\text{max}} - 2 \nu U_{\delta, 1}(0) \nonumber \\
&= (1-\delta)\nu^2 + \nu^2 - 2 \nu^2 + 2\epsilon_{\delta,1} \nu = 2\epsilon_{\delta,1}\nu - \delta \nu^2 \xrightarrow{\delta \to 0} 0.
\end{align}
Finally, the statement
\begin{align}
    c^2_{\delta} \geq \frac{\Pcal(W_{\delta})}{2 \Kcal(W_{\delta})} \xrightarrow{\delta \to 0} \infty
\end{align}
is a direct consequence of Lemma \ref{eq:estimate c} and the divergence of $\Pcal(W_{\delta})$.
\end{proof}

\bibliographystyle{alpha}

\end{document}